\documentclass[a4paper,10pt]{article}
\usepackage[utf8]{inputenc}
\usepackage{amsthm}
\usepackage{amsmath}
\usepackage{amsfonts}
\usepackage{amssymb}
\newtheorem{Lemma}{Lemma}
\newtheorem{Theorem}{Theorem}

\newtheorem{Definition}{Definition}
\newtheorem{assumption}{Assumption}

\title{On Parallel Transport in Wasserstein Space}
\author{Xi Sisi Shen}

\begin{document}

\maketitle

\begin{abstract}
In this short note, we would like to give a construction of parallel transport for tangent cones lying in the interior of a geodesic in Wasserstein space.
We give a complete proof for the linear part of the tangent space, and show that a construction for the full tangent cones follows from some natural lemmas on Wasserstein space.
It can easily be shown that our construction is equivalent to those used in the previous literature on this subject.
\end{abstract}

\section{Introduction}
An optimal transport plan between two measures $\mu_1$ and $\mu_2$ on a manifold $N$ is given by a 
probability distribution of geodesic segments on the manifold, such that the projection of this measure onto the first endpoint gives $\mu_1$ and projection onto the second endpoint gives $\mu_2$.
The fact that it is optimal means that the total integral of the squares of the lengths of the geodesics is minimal with this requirement.
Points in the Wasserstein space of a compact manifold $N$ are measures supported on the manifold and geodesics in Wasserstein space correspond to optimal transport maps on the manifold. 
Many important facts about Wasserstein space are given in \cite{villani}. 
In particular, it is straightforward to show that any two paths in an optimal transport plan cannot cross in the interior of the transport.
We will denote the Wasserstein distance by $W_{2}$.

\begin{Definition}
We define the distance between two transport plans emanating from a given measure $\mu$ to be the integral over all points of the Wasserstein distance between the two distributions each point is sent to.
\end{Definition}
The distance between two transport plans is evidently at least the distance between the endpoints of the transport plans as measures. 
It seems likely that equality is true in the limit for optimal transport plans as the lengths of the plans is scaled to $0$ (by shortening each geodesic in the probability distribution), and the distances are normalized; 
though I presently cannot give a proof or a reference for this.

We will say a geodesic is Monge (alternatively, Monge on one endpoint) if the projection of the geodesics onto both endpoints (on one endpoint) is one-to-one for almost all geodesics. 
Since transport paths cannot cross in the interior, any segment of a geodesic that lies in the interior of a Wasserstein geodesic is Monge. 
For any measure on the manifold $N$, we can naturally consider the set of Wasserstein geodesics one of whose endpoints is that measure. 
We can define the tangent cone to be the product of the set of maximal geodesic segments emanating from this measure with $R^{\geq 0}$.
We define the distance between two elements of the tangent cone to be the $\limsup$ as $\epsilon$ tends to $0$
of $\frac{1}{\epsilon} W_{2}(\gamma_1 (\epsilon r_{1}),\gamma_2 (\epsilon r_{2}))$. 
This makes the tangent cone into a metric cone; a metric cone is simply a metric space equipped with multiplication by $R^{\geq 0}$ in a way that commutes with the metric. Hence it possesses a unique $0$ element and we can talk
about the unit ball as the set of elements of the metric cone that lie a unit distance away from this $0$ element.
Alternatively, define the ``tangent set'' to be the set of all geodesics emanating from $\mu$ with the distance between two geodesic segments given by shortening each geodesic segment by a factor of $\epsilon$, 
and dividing the Wasserstein 
distance between the resulting endpoints by $\epsilon$, then taking the $\limsup$ of the result as $\epsilon$ goes to $0$. Now if we have a metric space equipped with multiplication by $[0,1]$ that commutes with the metric structure,
we can naturally extend it to a metric cone. This way, not all elements of the tangent cone correspond to a geodesic in the ``tangent set'', but only if we multiply the element by a small enough positive number.

In this short paper, we will attempt to construct a map between
the tangent cones along a Monge geodesic that would correspond to the Wasserstein parallel transport along that geodesic.
Our motivation mainly comes from the results in \cite{Lott}, where parallel transport is constructed for a special and rather narrow class of
Wasserstein geodesics. We will give a complete construction for the linear tangent space, which is the subset of geodesic segments that can be extended in both directions as a minimizing geodesics, and we will show that a 
construction for the complete tangent cone follows from some natural lemmas on Wasserstein space. In both cases, the families of maps we construct are isometries between the appropriate spaces.

\section{Construction for the Linear Tangent Space}

\subsection{Definitions and Notations}

We will define a family of operators, closed under composition, that allow us to define parallel transport of the linear tangent space along a Monge geodesic $\mu_{t}$, 
parametrized proportionally to length from $0$ to $1$, in the Wasserstein space of a manifold.
Every element in the linear tangent space $T\mu_{t}$ at $\mu_{t}$ is a vector field in the Hilbert space $L^2(supp(\mu_{t}),\mu_{t})$. We will denote the normal space of $\mu_{t}$ in this Hilbert space
by $N\mu_{t}$, which is the closure of the space of vector fields that preserve $\mu_{t}$. For any $t_1,t_2$, we will denote the operators given by the usual parallel transport along the Monge geodesic by 
$ParT_{t_1,t_2}$ and the operator corresponding to pushforward, whenever it is defined, by $Push_{t_1,t_2}$. We note immediately that $ParT$ commutes with the Hilbert space structure of the vector
fields in $L^2(supp(\mu_{t}),\mu_{t})$, that the pushforward is always defined on any element in $N\mu_{t_1}$ and furthermore that $Push_{t_1,t_2}N\mu_{t_1}=N\mu_{t_2}$, as can be seen from
the definition of $N\mu$. Finally we will use the fact that, the difference, as operators, of $Push_{t_1,t_2}-ParT_{t_1,t_2}$ is bounded in operator norm, near $t_1$ by the first derivative of the Monge transport field.
From this, and the unitarity of parallel transport, it easily follows that the difference is bounded, globally, by $C W(\mu_{t_1},\mu_{t_2})$ where $C$ is a global constant that depends only on the first derivative
 of the Monge transport field.  

\subsection{Proof}

We consider families of operators $F_{p_i,p_j}$ defined between the spaces of vector fields in $L^2(supp(\mu_{p_i}),\mu_{p_i})$, for $\{p_{i}\}_{i=1}^{n}$ ranging on a subset of the unit interval.
We let $T$ denote the family of operators defined by $T_{p_i,p_j}=proj_{T\mu_{p_j}}ParT_{p_i,p_j}$. Given a function $F$ from the non-negative reals to the non-negative reals, we say that a family $S$
is an $F$-approximation of $T$, if for any two points $q_i$, $q_j$ between which $S_{q_i,q_j}$ is defined, 
\begin{equation*}
||S_{q_i,q_j}-T_{q_i,q_j}||\leq F(|q_j-q_i|)
\end{equation*}
in operator norm.

For any finite family of operators $S$, defined between points in a finite subset of the unit interval, we can define the \textit{homogenization} of $S$ to be the family of operators obtained by composing the operators
associated to neighboring points. The result will be a family of operators satisfying the composition property. We would like morally to define a homogenization of $T$. We will do so as follows.
\begin{Lemma}
There exists an $F$ with the property that $F(t) \leq o(t)$ s. t. the homogenization of the restriction of $T$ to any finite subset of the unit interval is an $F$-approximation of $T$.
\end{Lemma}

\begin{proof}
Observe first that the parallel transport operator $ParT$ is unitary and that the pushforward preserves the normal space: $Push_{q_1,q_2}{N\mu_{q_1}}={N\mu_{q_2}}$.
Let $p_1$, $p_2$, $p_3$, ... , $p_3$  be points in the unit interval. First we bound in operator norm $||ParT_{p_i,p_j}-T_{p_i,p_j}||$. We multiply this operator by its adjoint to get:
\begin{equation*}
||ParT_{p_i,p_j}-T_{p_i,p_j}||^2=
||(ParT_{p_i,p_j}-T_{p_i,p_j})^{adj}(ParT_{p_i,p_j}-T_{p_i,p_j})||_{T\mu_{p_i}}= 
\end{equation*}
\begin{equation*}
||(ParT_{p_i,p_j})^{adj}(ParT_{p_i,p_j}-T_{p_i,p_j})||_{T\mu_{p_i}}=
||ParT^{-1}_{p_i,p_j}(ParT_{p_i,p_j}-T_{p_i,p_j})||_{T\mu_{p_i}}=
\end{equation*}
\begin{equation*}
||(ParT^{-1}_{p_i,p_j}-Push^{-1}_{p_i,p_j})(ParT_{p_i,p_j}-T_{p_i,p_j})||_{T\mu_{p_i}}\leq
\end{equation*}
\begin{equation*}
||(ParT^{-1}_{p_i,p_j}-Push^{-1}_{p_i,p_j})||_{T\mu_{p_i}} ||(ParT_{p_i,p_j}-T_{p_i,p_j})||_{T\mu_{p_i}}
\end{equation*}
But $||(ParT^{-1}_{p_i,p_j}-Push^{-1}_{p_i,p_j})||\leq C W(\mu_{p_i},\mu_{p_j})$, where $C$ is a universal constant depending only on the Monge transport field. Hence the same bound holds for $||(ParT_{p_i,p_j}-T_{p_i,p_j})||$ and 
\begin{equation}
 ||(ParT_{p_i,p_j}-T_{p_i,p_j})|| \leq C W(\mu_{p_i},\mu_{p_j})
\end{equation}

Next we observe
$||T_{p_i,p_j}T_{p_j,p_k}-T_{p_i,p_k}||=||T_{p_i,p_j}(ParT_{p_j,p_k}-T_{p_j,p_k})||=||Proj_{T\mu_{p_i}}ParT_{p_i,p_j}(ParT_{p_j,p_k}-T_{p_j,p_k})||=
||Proj_{T\mu_{p_i}}(Push_{p_i,p_j}-ParT_{p_i,p_j})(ParT_{p_j,p_k}-T_{p_j,p_k})||$
$\leq||(Push_{p_i,p_j}-ParT_{p_i,p_j})||||(ParT_{p_j,p_k}-T_{p_j,p_k})||\leq C^2 W(\mu_{p_i},\mu_{p_j})W(\mu_{p_j},\mu_{p_k})$ 
\newline
\newline
where in the last line we used inequality (1).
\newline
\newline
Finally,
$||\prod\limits^{n-1}_{i=1} T_{p_{i+1},p_{i}}-T_{p_n,p_1}||=||T_{p_{n},p_{n-1}}T_{p_{n-1},p_{n-2}}\prod\limits^{n-3}_{i=1} T_{p_{i+1},p_{i}}-T_{p_n,p_1}||\leq
||T_{p_{n},p_{n-2}}\prod\limits^{n-3}_{i=1} T_{p_{i+1},p_{i}}-T_{p_n,p_1}||+||T_{p_n,p_{n-1}}T_{p_{n-1},p_{n-2}}-T_{p_n,p_n-2}||||\prod\limits^{n-3}_{i=1} T_{p_{i+1},p_{i}}||
\leq||T_{p_{n},p_{n-2}}\prod\limits^{n-3}_{i=1} T_{p_{i+1},p_{i}}-T_{p_n,p_1}||+||T_{p_n,p_{n-1}}T_{p_{n-1},p_{n-2}}-T_{p_n,p_n-2}||$
\newline
\newline
Iterating the last inequality, we obtain 
\begin{equation*}
  ||\prod\limits^{n-1}_{i=1} T_{p_{i+1},p_{i}}-T_{p_n,p_1}||\leq \sum_{i=1}^{n-2}||T_{p_n,p_{i+1}}T_{p_{i+1},p_{i}}-T_{p_{n},p_i}||
\end{equation*}

Applying the expression for the bound of $||T_{p_i,p_j}T_{p_j,p_k}-T_{p_i,p_k}||$ and for $||(Push_{p_i,p_j}-ParT_{p_i,p_j})||$, we get:
\begin{equation}
 ||\prod\limits^{n-1}_{i=1} T_{p_{i+1},p_{i}}-T_{p_n,p_1}||\leq C^2 W^2(\mu_{p_n},\mu_{p_1}) = C^2 W^2(\mu_{1},\mu_{0}) (p_n-p_1)^2
\end{equation}
which shows indeed that the homogenization of the restriction of $T$ to $\{p_i\}_{i=1}^n$ is an $F$-approximation of $T$, for $F(t)=C^2 W^2(\mu_{1},\mu_{0}) t^2$ and any finite subset $\{p_i\}_{i=1}^n$.
\end{proof}

Now suppose, for any $F$, a finite homogenous (satisfying the composition property) family of operators $S_1$ is an $F$-approximation of $T$ on some finite set of points $\{q_i\}_{i=1}^n$. Let us define the width of this set with respect to $F$ by 
\begin{equation}
 w^F(\{q_i\}_{i=1}^n)\, 
=exp\Big(\sum_{i=1}^{n-1}F(|q_{i+1}-q_i|)\Big)-1  
\end{equation}

Then every operator in $S_1$, in operator norm, lies within a ball of radius $w^F(\{q_i\}_{i=1}^n)$ of the homogenization of the restriction of $T$ to $\{q_i\}_{i=1}^n$.

\begin{proof}
This follows from the triangle inequality and the fact that the norm of every operator in $T$ is $1$, after expanding the expression for the operator in $S_1$ between two points, as a composition of operators on neighboring points,
and applying the fact that it is an $F$-approximation of $T$.
\end{proof}

Now observe that for any increasing sequence of subsets $\{q^k_i\}_{i=1}^{n_k}$ of the unit interval that increases to a dense set, the $F$-widths  $w^F_k$ of the subsets converge to $0$, for any $F(t)\leq o(t)$.
To every subset in this increasing sequence we associate the homogenization of the restriction of $T$. This will give an increasing sequence of homogenous families of operators. For any two points in the limiting subset of the unit interval, $p$ and $q$,
this will give a sequence of operators $S^i_{p,q}$. By the above, the sequence of operators will satisfy $||S^{k_2}_{p,q}-S^{k_1}_{p,q}||\leq w^F_{min(k_1,k_2)}$. Hence for any two points in the limiting sequence,
 one will obtain a convergent
sequence of operators. The convergence will be uniform. The limiting operators will provide an $F$-approximation of $T$ for any two points on which they are defined and will satisfy the composition property. Furthermore, for any 
given dense sequence, it is an easy consequence of the estimate that such a family is unique. Thus, for any two points in the unit interval, the limiting operator thus obtained, will be independent of the increasing sequence of finite subset, or the 
limiting dense subset; as can be seen from taking the union of two increasing sequences of subsets, increasing to different sets.

Finally it will follow from the fact that the operators in $T$ are locally unitary, that every operator in the limiting family of operators will be unitary. We observe finally that in \cite{Lott} a different family $\tilde{T}$ was considered. However
this family agrees with our family up to $O(t^2)$ and hence also produces an $F$-approximation to our family of operators $T$, after restricting to a sequence and homogenizing, from which we can deduce
 that it converges to the same limit.

\section{Construction for the Full Tangent Cone}
For any two points $\mu_{s}$ and $\mu_{t}$ on a Monge geodesic, we can define a map between their associated tangent cones $T_{\mu_{s}}$ and $T_{\mu_{t}}$ as follows. 
An element of the tangent cone at $T_{\mu_{s}}$ is a maximal geodesic segment and a real number $r$. 
Choose a small enough $\epsilon$ such that this geodesic segment has a proper closed geodesic subsegment $\gamma_{\epsilon r}$ of length $\epsilon r$. 
This segment defines an optimal transport plan to some measure $\nu_{s}$
in the Wasserstein space of $N$. Since we are in a Monge geodesic, and since $\gamma_{\epsilon r}$ corresponds to a probability
distribution of geodesic segments one of whose endpoints lies in $supp(\mu_s)$, we can parallel transport $\gamma_{\epsilon r}$, according to the ordinary Riemannian parallel transport on $N$, along the Monge geodesic.
This gives some new measure $\nu_{t}$ equipped with a transport map to $\mu_{t}$, which however may no longer be an optimal transport. We can consider the optimal transport map from $\mu_{s}$ to $\nu_{s}$. This will give
a geodesic segment of length $W_{2}(\mu_{t},\nu_{t})$ (at most $\epsilon r$, since the ordinary parallel transport preserves the length of a transport plan), which lies in some maximal geodesic segment.
Taking this maximal geodesic segment and taking the associated real number to be the ratio of the length of $W_{2}(\mu_{t},\nu_{t})$ to $\epsilon$ gives us an element of the tangent cone at $\mu_{t}$. 
The question of whether this gives a well-defined element of the tangent cone as $\epsilon$ is taken to $0$ is not fully evident. 
However it is, if it is the case that any transport plan emanating from a measure can always be approximated by an optimal transport for short time.
So, when its image is well-defined, we will refer to this map from $T_{\mu_{s}}$ to $T_{\mu_{t}}$ as $M_{s,t}$. This map will be the basis of our construction.



\begin{assumption}
$M_{s,t}$ is well-defined as a map between tangent cones.
\end{assumption}
We will leave this lemma unproven. However, it is obvious at least in particular cases and should be true in general.

Moreover, we will also have the following unproven lemma:

\begin{assumption}

$M_{s,t}$ is non-expanding as a map between metric spaces.
\end{assumption}
This should also be a natural consequence of basic facts on Wasserstein space and is true in particular cases. In particular, it would follow from the fact that the distance between two short Wasserstein
geodesics lying in the tangent set is equal to the distance between them as transport plans, as defined in the introduction, in the normalized $0$ scaling limit.

\begin{Definition}
Define the operator $M_{s,t}$ to be the above map between from the tangent cone to $\mu_{s}$ to the tangent cone to $\mu_{t}$.
\end{Definition}
 
Now if well-defined, this gives a family of maps between any two tangent cones in the interior of a Wasserstein geodesic. However this family of maps does not obey the natural composition property.
We will use this family of maps to construct Wasserstein parallel transport between two points in the interior of the geodesic,
by subdividing the interval between the two points into small segments and composing the operators associated to each
segment, then taking the subdivision increasing to a dense subdivision of the interval between the two points, 
and showing that the result gives a convergent sequence on every tangent cone associated to a point in the dense subdivision. For a given subdivision $S$, we will denote the family of maps between
 the endpoints of the subdivision $\{ s_{i} \}_{i=1}^{n}$, arising in this way by $M(S)_{s_{i},s_{j}}$. Roughly speaking, we will take $S$ increasing to a dense subdivision and show that for any two points $t_1,t_2$ in the dense
subdivision, $M(S)_{t_{1},t_{2}}$ converges uniformly to an isometry as $S$ approaches the dense subdivision. Convergence of maps between metric cones here refers to uniform convergence of the images of the unit ball.

\subsection{Proof}

In order to prove convergence, it will be crucial to bound the discrepancy between $M_{t_3,t_2} \circ M_{t_2,t_1} $ and $M_{t_3,t_1}$, acting on a unit element of $T\mu_{t_1}$. This will be the core of the proof.
Thus, we have the following theorem.
\begin{Theorem}
\label{thm1}
For any unit element in the tangent cone at $\mu_{t_1}$, the distance between the images under $M_{t_3,t_2} \circ M_{t_2,t_1} $ and $ M_{t_3,t_1}$ of this element in $T\mu_{t_3}$ is uniformly bounded by 
$C \,  W_{2}(\mu_{t_2},\mu_{t_3}) \,  D(t_1,t_2)$, where $C$ is a constant, that depends on the 
manifold $N$ and the Wasserstein geodesic, and where the quantity $D(t_1,t_2)$ will be defined in the course of the proof.
\end{Theorem}
Take a unit element of $T\mu_{t_1}$ and choose an $\epsilon$ to obtain a geodesic segment of length $\epsilon$ emanating from $\mu_{t_1}$ and joining it to some measure $\nu_{1}$. Parallel transport the transport map joining
$\mu_{t_1}$ and $\nu_{1}$ to $\mu_{t_2}$, to obtain a measure $\nu_{2}$, equipped with a transport map to $\mu_{t_2}$. This transport map $P_{1}$ associates each point in the support of $\nu_{2}$ to an element in the support of
$\mu_{t_2}$, possibly non-uniquely. Consider the optimal transport plan from $\nu_{2}$ to $\mu_{t_2}$. This will give another transport plan $P_{2}$ from $\nu_{2}$ to $\mu_{t_2}$. 

We can consider the distance between the two plans $P_{1}$ and $P_{2}$ in the sense defined above. In particular, the ratio of this distance to $\epsilon$ as $\epsilon$
is taken to $0$ will be an important parameter in the proof, and will, in a sense measure the angle between $P_{1}$ and the measure. 
It will be important for us to have a mild bound on this distance in order to prove convergence. 

Applying parallel transport to the first transport plan now yields a measure $\nu_3$ along with a transport plan to $\mu_{t_3}$, and applying parallel transport to the second transport plan yields a measure  $\tilde{\nu}_3$.
In order to bound the resulting discrepancy between the projections onto the tangent cone $T\mu_{t_3}$, it is sufficient to bound the Wasserstein distance between $\tilde{\nu}_3$ and $\nu_3$.
This is fairly straighforward, given a bound on the distance between the above two transport plans. Every point in the support of $\nu_2$ travels along a probability distribution of curves to the supports of $\tilde{\nu}_3$
and $\nu_3$.
These curves are not geodesics, in general. However they are in Euclidean space; moreover, in that case, the directions of these geodesics are given by parallel transporting, along $P_{1}$ and $P_{2}$, the vector field
that gives the instaneous transport $\mu'_{t_2}$ of $\mu_{t_2}$ along the Wasserstein geodesic $\mu_{t}$. Replacing the curves by these geodesics in the non-Euclidean case will change the result by an amount that is
uniformly quadratic in 
$W_2(\mu_{t_1},\nu_1)=\epsilon$, and will hence vanish in the limit when we take $\epsilon$ to $0$. The discrepancy between $\tilde{\nu}_3$ and $\nu_3$ can hence be measured by the discrepancy of the transport plans corresponding to
these two sets of distributions of geodesics. The first set is obtained by parallel transporting $\mu'_{t_2}$ along the first transport plan, the second by transporting $\mu'_{t_2}$ by the second transport plan. 
We can compute this discrepancy, as follows. We have two sets of probability distributions of geodesics. The first arises by parallel transporting the geodesics that form the infinitesimal transport $\mu'_{t_2}$ by $P_{1}$; 
the second arises by parallel transporting this same set by $P_{2}$. Each point $p$ in the support of $\nu_2$ is mapped to a probability distribution of points in $\mu_{t_2}$ by $P_{1}$ and to another distribution by $P_{1}$. 
Each of these distributions will give a distributions of vectors in the tangent space at $p$. The optimal transport plan between the two distributions of points will lift to a transport plan between the two distributions of vectors.
The optimal transport plan between the two distributions of points is a distribution of geodesic segments each of which we lift to a difference of two vectors in $T_{p}N$. Let us understand in detail the lifting of each segment
The two vectors from the lifting will correspond to the evaluation of the vector field
$\mu'_{t_2}$ at one endpoint of the geodesic segment, parallel transported to $p$, and the evaluation of $\mu'_{t_2}$ at the other endpoint of the segment, parallel transported to $p$. We can evaluate this difference at $p$;
we can also evaluate this difference after parallel transporting both vectors to one of the endpoints of the segment, and then along the segment; this is true because parallel transport is norm-preserving.
Hence, one of the vectors will have travelled all around the geodesic triangle, and
the other vector will have come back to its point of origin and then gotten parallel transported to the other endpoint of the segment. 
This shows that we can bound the difference between these two vectors by $(1)$ the failure of the
evaluations of $\mu'_{t_2}$ at the two endpoints to be parallel transports of each other along the segment, and $(2)$ the deviation resulting from the parallel transport of one of the vectors all around the geodesic triangle. 
The second quantity  $(2)$ is controlled by a constant that depends on the curvature of the manifold $N$, multiplying the area of the geodesic triangle in the construction. After summing over all geodesic triangles and all $p$,
this curvature term will be quadratic in $W_2(\mu_{t_1},\nu_1)$,
and will hence vanish in the limit when we take $\epsilon$ to $0$, and divide by $\epsilon$.
The first term  $(1)$ is the intergral over all geodesic segments joining $P_1$ and $P_2$ of the covariant derivative of the vector field $\mu'_{t_2}$; 
this yields the sum of the lengths of the geodesic segments multiplied by a quantity that
uniformly controls the variation of $\mu'_{t_2}$, which we can take to be a universal constant.
After summing the lengths of all the geodesic segments, we obtain the Wasserstein distance between the two transport maps.
Thus, there exists a uniform constant that depends only on $\mu_{t}$ and on the manifold
such that the Wasserstein distance between $\tilde{\nu}_3$ and $\nu_3$ is bounded above by this constant, multiplied by the distance between $\mu_{t_2}$ and $\mu_{t_3}$ and the discrepancy
between the two transport plans joining $\mu_{t_2}$ and $\nu_{2}$, up to terms that are quadratic in $\epsilon$.
Dividing the result by $\epsilon$ shows that the discrepancy between the images of $M_{t_3,t_2} \circ M_{t_2,t_1}$
and $M_{t_3,t_1}$, is controlled up to absolute constants, by the product of $W_2(\mu_{t_2},\mu_{t_3})$ and the limit of the discrepancy of the transport plans, $W_{2}(P_{1},P_{2})$ divided by $\epsilon$. 
We will call the latter quantity ($\limsup \frac{1}{\epsilon} W_{2}(P_{1},P_{2})$) $D(t_1,t_2)$ after taking the supremum over all unit elements in the tangent cone at $\mu_{t_1}$. 
Hence we can conclude that the discrepancy between $\tilde{\nu}_3$ and $\nu_3$ and hence between the images of $M_{t_3,t_2} \circ M_{t_2,t_1}$
and $M_{t_3,t_1}$ is that given in the theorem.
\begin{Definition}
For any segment $t_1,t_2$, define $\bar{D}(t_1,t_2)$ to be the maximum of $D(s_1,s_2)$ over all $\{ s_1,s_2 \}$ lying between $t_1$ and $t_2$.
\end{Definition}

The rest is straightforward. Knowing the above discrepancy and using the fact that all maps $M_{s,t}$ are non-expanding, we can determine the size of the discrepancy between 
$M_{t_1,t_n}$ and $\prod\limits_{i=1}^{n-1} M_{t_{n-i},t_{n-i+1}}$, by taking out one $t_i$ at a time. If we start from $t_{n-1}$ and continue until we remove $t_2$, the error coming from each step, will be
 $W_2(\mu_{t_i},\mu_{t_{i+1}})$
multiplied by $D(t_1,t_i)$, multiplied by the universal constant. Thus the total error will be a product of $\bar{D}(t_1,t_n)$, $W_2(\mu_{t_1},\mu_{t_{n}})$ and the universal constant. This quantity does not depend
on the intermediate points and we will call it the total potential error associated with a segment.

\begin{Lemma}
Now let $S$ be a subdivision of the interval from $s$ to $t$ and let $T_{1}$ and $T_{2}$ be two other subdivisions containing $S$. After taking iteratively the compositions of $M_{\cdot,\cdot}$ along $T_1$ and $T_2$,
the discrepancy between the two resulting images of any unit vector on a tangent cone associated to $S$ will be bounded above by twice the sum of the total potential errors over all segments of $S$. 
Call this the width of the subdivision. Notice that if we add more points to a subdivision, its width cannot increase.
\end{Lemma}
The proof is obvious.

We say that a dense subdivision has $0$ width if there exists a sequence of finite subdivisions that increase to it that have widths tending to $0$. The requirement of $0$ widths however, by the above monotonicity observation does
not depend on the increasing subsequence. 
\begin{assumption}
\label{assn3}
We will assume that the interval $[s,t]$ on which we wish to construct parallel transport admits a dense subdivision of $0$-width.
\end{assumption}
This would follow from some weak estimates, that I however at present cannot show.

Now take an interval $[\mu_{s}, \mu_{t}]$, lying inside a Wasserstein geodesic. Take a family of subdivisions of $[s,t]$ increasing to a dense subdivision. To every subdivision associate a family of maps between any two tangent cones associated
to points in the subdivision by composing elements in the family of maps $M$. The discrepancy between the maps arising from any two subdivisions, when applied to unit elements will be bounded above by the width of any
earlier subdivision. Thus if the limit of the widths is $0$, this yields uniformly Cauchy sequences on each tangent cone associated to a point in the dense subdivision. In this case, again,
we say that the limiting dense subdivision has $0$ width. Moreover, if we have two subdivisions that yield a limit in this way, taking their union shows that the limits must coincide. 
Hence on any point shared by two subdivisions with $0$ width, the constructions of parallel transport arising from the two subdivisions will coincide. 
Hence the construction yields a well-defined parallel transport map between any two points lying in a segment $(\mu_s,\mu_t)$ contained in the interior of a Wasserstein geodesic.
Naturally, if there exists at least one dense subdivision of $0$ width, we can add any point to it without increasing its width, and hence every point inside the segment will lie in a dense subdivision of $0$ width.
Thus, assuming there exists at least one dense subdivision of $0$ width of a segment, there exists a uniquely defined parallel transport map in the interior of that segment. 
\newline
\newline

Since it is constructed as a limit of
non-expanding maps, the resulting Wasserstein parallel transport map must be non-expanding. Moreover, if we have at least one subdivision of $0$ width,
we can always construct parallel transport maps in both directions. Noticing that Theorem \ref{thm1} continues to hold when we substitute $t_3=t_1$, we can see that the composition of the forward and backward transport maps
must yield the identity in the limit, i.e.:

\begin{Lemma}
 The Wasserstein parallel transport map constructed above is an isomorphism, whose inverse can be obtained by applying the same construction in the reverse direction.
\end{Lemma}

Here is a detailed proof:

\begin{proof}
Theorem \ref{thm1} implies that the distance between any element in the unit ball of the tangent cone at $\mu_{t_1}$ and its image under $M_{t_2,t_1} \circ M_{t_1,t_2}$ is bounded above by
 $C W_{2}(\mu_{t_1},\mu_{t_2}) D(t_1,t_2)$.  Hence the distance between the images of such an element under the maps $\prod\limits_{i=1}^{n-1} M_{t_{n-i},t_{n-i+1}} \circ \prod\limits_{i=1}^{n-1} M_{t_{i},t_{i+1}}$ and
 $\prod\limits_{i=1}^{n-2} M_{t_{n-i},t_{n-i+1}} \circ \prod\limits_{i=1}^{n-2} M_{t_{i},t_{i+1}}$ is bounded above by $C W_{2}(\mu_{t_{n-1}},\mu_{t_{n-2}}) D(t_{n-1},t_{n-2})$. Hence,
the distance between $\prod\limits_{i=1}^{n-2} M_{t_{n-i},t_{n-i+1}} \circ \prod\limits_{i=1}^{n-2} M_{t_{i},t_{i+1}}$ and the identity is bounded by the width of the partition $\{t_i\}_{i=1}^{n}$. Using this fact,
 suppose we have a segment $[s,t]$,
lying inside a Wasserstein geodesic, and we use a dense subdivision  (of $0$ width) of this segment to construct an element $\gamma_2$ in the tangent cone of $\mu_t$ from an element $\gamma_1$ in the unit tangent cone to $\mu_s$.
Then a finite subdivision $S$ of the interval that approximates the dense subdivision, will yield an operator that approximates the Wasserstein parallel transport arising from the dense subdivision in both directions. Hence,
in the forward direction, the image of $\gamma_{1}$ under this operator will be approximately $\gamma_{2}$, with an error $Err_{1}$; 
moreover, the distance between $\gamma_1$ and its image under the composition of the forward and reverse operators constructed from $S$
will be bounded by the width of the finite subdivision, with an error $Err_{2}$. Thus the deviation of $\gamma_1$ from being the image of $\gamma_2$ under the reverse operator associated to the finite subdivision will be bounded
by the sum of $Err_{1}$ and $Err_{2}$. As the finite subdivision is taken to approach the dense subdivision, both errors will tend to $0$.
\end{proof}

Finally, since the forward and reverse parallel transport maps we have constructed are inverses of each other, and since both are non-expanding, both must be isometries. 
Again, this is all conditional on there existing a subdivision of width
$0$; however again, this would follow from very weak estimates on the quantity $D(t_1,t_2)$; it would certainly follow from the fact that $D(t_1,t_2)$ uniformly goes to $0$ as $W(\mu_{t_1},\mu_{t_2})$ 
goes to $0$.


\end{document}